\documentclass[10pt,reqno]{amsart}

\numberwithin{equation}{section}

\newtheorem{THM}{Theorem}[section]
\newtheorem{LMA}[THM]{Lemma}
\newtheorem{PROP}[THM]{Proposition}
\newtheorem{CORO}[THM]{Corollary}

\usepackage{amsmath}
\usepackage{amssymb}
\usepackage{euscript,url}

\newcommand{\showon}{\begin{eqnarray*}}
\newcommand{\showoff}{\end{eqnarray*}}

\newcommand{\diag}{\mathrm{diag}}
\newcommand{\goesto}{\rightarrow}
\newcommand{\one}{\boldsymbol{1}}
\newcommand{\zero}{\boldsymbol{0}}
\newcommand{\per}{\mathrm{per}}
\renewcommand{\Im}{\mathrm{Im}}
\renewcommand{\i}{\mathrm{i}}
\newcommand{\cyc}{\mathrm{cyc}}
\newcommand{\ma}{\mathsf{ma}}

\newcommand{\A}{\EuScript{A}} \renewcommand{\a}{\mathbf{a}}

\newcommand{\C}{\EuScript{C}} \newcommand{\CC}{\mathbb{C}}

\newcommand{\D}{\EuScript{D}}

\newcommand{\HH}{\EuScript{H}}

\newcommand{\M}{\EuScript{M}}

 \newcommand{\RR}{\mathbb{R}}
 \renewcommand{\S}{\mathfrak{S}}

\renewcommand{\v}{\mathbf{v}}
\newcommand{\w}{\mathbf{w}}
 \newcommand{\x}{\mathbf{x}}

 \newcommand{\y}{\mathbf{y}}

\newcommand{\z}{\mathbf{z}}

\begin{document}

\title[Proof of the MCP Conjecture]{Proof of the monotone column \\
permanent conjecture}

\author[Br\"and\'en]{Petter Br\"and\'en}
\thanks{P.B. is a Royal Swedish Academy of Sciences Research Fellow
supported by a grant from the Knut and Alice Wallenberg Foundation.}
\address{Department of Mathematics,\ Stockholm University,\ SE-106 91\
Stockholm,\ Sweden}
\email{pbranden@math.su.se}

\author[Haglund]{James Haglund}
\thanks{Research of J.H. is supported by NSF grants DMS-0553619
and DMS-0901467}
\address{Department of Mathematics,\ University of Pennsylvania,\
Philadelphia, PA, USA,\ 19104}
\email{jhaglund@math.upenn.edu}

\author[Visontai]{Mirk\'o Visontai}
\thanks{Research of M.V. is supported by the Benjamin Franklin Fellowship of the University of Pennsylvania.}
\address{Department of Mathematics,\ University of Pennsylvania,\
Philadelphia, PA, USA,\ 19104}
\email{mirko@math.upenn.edu}

\author[Wagner]{David G. Wagner}
\thanks{Research of D.G.W. is supported by NSERC Discovery Grant OGP0105392.}
\address{Department of Combinatorics and Optimization,\
University of Waterloo,\ Waterloo, Ontario, Canada,\ \ N2L 3G1}
\email{dgwagner@math.uwaterloo.ca}


\dedicatory{In memoriam Julius Borcea}

\begin{abstract}

Let $A$ be an $n$-by-$n$ matrix of real numbers which are weakly decreasing down each 
column, $Z_n = \text{diag}(z_1,\ldots, z_n)$ a diagonal matrix of indeterminates, and  $J_n$  
the $n$-by-$n$ matrix of all ones.  We prove that $\text{per}(J_nZ_n+A)$ is stable in the $z_i$, 
resolving a recent conjecture of Haglund and Visontai.  This immediately implies 
that $\text{per}(zJ_n+A)$ is a polynomial in $z$ with only real roots, an open conjecture of 
Haglund, Ono, and Wagner from $1999$.  Other applications include a multivariate stable 
Eulerian polynomial, a new proof of Grace's apolarity theorem and new permanental inequalities.

\end{abstract}

\maketitle

\section{The monotone column permanent conjecture.}

Recall that the \emph{permanent} of an $n$-by-$n$ matrix $H=(h_{ij})$
is the unsigned variant of its determinant:
$$
\per(H) = \sum_{\sigma\in\S_n} \prod_{i=1}^n h_{i,\sigma(i)},
$$
with the sum over all permutations $\sigma$ in the symmetric group $\S_n$.
A \emph{monotone column matrix} $A=(a_{ij})$ has real entries which are
weakly decreasing reading down each column:\ that is, $a_{ij}\geq a_{i+1,j}$
for all $1\leq i\leq n-1$ and $1\leq j\leq n$.  Let $J_n$ be the $n$-by-$n$ 
matrix in which every entry is $1$.\\

\noindent\textbf{The Monotone Column Permanent Conjecture (MCPC).}\\
If $A$ is an $n$-by-$n$ monotone column matrix then $\per(zJ_n+A)$
is a polynomial in $z$ which has only real roots.\\

The MCPC first appears as Conjecture 2 of \cite{HOW}.
(Originally with increasing columns -- but the convention of decreasing
columns is clearly equivalent and will be more natural later.)
Theorem 1 of \cite{HOW} proves the MCPC for monotone column matrices
in which every entry is either $0$ or $1$.  Other special cases appear in
\cite{H,HOW,HV,KYZ}, either for $n\leq 4$ or for rather restrictive conditions
on the entries of $A$.  In this paper we prove the MCPC in full
generality\footnote{Thus, it was not really a permanent conjecture.}.
In fact we prove more.  The following multivariate version of the MCPC
originates in \cite{HV}.  Let $Z_n=\diag(z_1,\ldots,z_n)$ be an $n$-by-$n$
diagonal matrix of algebraically independent commuting indeterminates
$\z=\{z_1,\ldots,z_n\}$.  A polynomial $f(z_1,\ldots,z_n)\in\CC[\z]$ is
\emph{stable} provided that either $f\equiv 0$ or whenever $w_j\in\CC$ are
such that $\Im(w_j)>0$ for all $1\leq j\leq n$, then $f(w_1,\ldots,w_n)\neq 0$.
A stable polynomial with real coefficients is \emph{real stable}.
\\

\noindent\textbf{The Multivariate MCP Conjecture (MMCPC).}\\
If $A$ is an $n$-by-$n$ monotone column matrix then $\per(J_nZ_n+A)$
is a real stable polynomial in $\RR[\z]$.\\

Note that $J_nZ_n+A=(z_j+a_{ij})$, so that $z_j$ is associated with the $j$-th
column, for each $1\leq j\leq n$.  We also write
$\per(J_nZ_n+A)=\per(z_j+a_{ij})$ as it seems clearer.  The MMCPC implies the
MCPC since if one sets all $z_j=z$, then $\per(z+a_{ij})$ is a polynomial
in one variable with real coefficients; this diagonalization preserves stability
(see Lemma \ref{basic}(c) below), and a univariate real polynomial is stable if and
only if it has only real roots.

In Section 2 we review some results from the theory of stable polynomials which
are required for our proofs.  
In Section 3 we reduce the MMCPC to the case of $\{0,1\}$-matrices which are
weakly decreasing down columns and weakly increasing from left to right across
rows -- these we call \emph{Ferrers matrices} for convenience.
Then we further transform the MMCPC for Ferrers matrices,
derive a differential recurrence relation for the resulting polynomials, and
use this and the results of Section 2 to prove the conjecture by induction.
In Section 4 we extend these results to sub-permanents of rectangular matrices,
derive a cycle-counting extension of one of them, discuss a multivariate stable
generalization of Eulerian polynomials,  present a new proof of Grace's
apolarity theorem and derive new permanental inequalities.

\section{Stable polynomials.}

Over a series of several papers, Borcea and Br\"and\'en have developed
the theory of stable polynomials into a powerful and flexible technique.
The results we need are taken from \cite{BB4,BB5,Br}; see also Sections 2
and 3 of \cite{W}.

Let $\HH=\{w\in\CC:\ \Im(w)>0\}$ denote the open upper
half of the complex plane, and let $\overline{\HH}$ denote its closure in $\CC$.
As above, $\z=\{z_1,\ldots,z_n\}$ is a set of $n$ indeterminates.
For $f\in\CC[\z]$ and $1\leq j\leq n$, let $\deg_{z_j}(f)$ denote the degree of $z_j$
in $f$.
\begin{LMA}[see Lemma 2.4 of \cite{W}] \label{basic}
These operations preserve stability of polynomials in $\CC[\z]$.\\
\textup{(a)}\ \textup{Permutation:}\ for any permutation $\sigma\in\S_n$,
$f \mapsto f(z_{\sigma(1)},\ldots,z_{\sigma(n)})$.\\
\textup{(b)}\ \textup{Scaling:}\ for $c\in\CC$ and $\a\in\RR^{n}$ with
$\a>\zero$, $f \mapsto cf(a_{1}z_{1},\ldots,a_{n}z_{n})$.\\
\textup{(c)}\ \textup{Diagonalization:}\ for $1\leq i<j\leq n$,
$f \mapsto f(\z)|_{z_i=z_j}$.\\
\textup{(d)}\ \textup{Specialization:}\ for $a\in\overline{\HH}$,
$f \mapsto f(a, z_{2},\ldots,z_{n})$.\\
\textup{(e)}\ \textup{Inversion:}\ if $\deg_{z_1}(f)=d$,
$f \mapsto z_{1}^{d}f(-z_{1}^{-1},z_{2},\ldots,z_{n})$.\\
\textup{(f)}\ \textup{Translation:}\
$f \mapsto f_1 = f(z_1+t,z_2,\ldots,z_n)\in\CC[\z,t]$.\\
\textup{(g)}\ \textup{Differentiation:}\
$f \mapsto \partial f(\z)/\partial z_1$.
\end{LMA}
\begin{proof}
Only part (f) is not made explicit in \cite{BB4,BB5,W}.  But clearly if
$z_1\in\HH$ and $t\in\HH$ then $z_1+t\in\HH$, from which the result follows.
\end{proof}
Of course, parts (d,e,f,g) apply to any index $j$ as well, by permutation.
Part (g) is the only difficult one -- it is essentially the Gauss-Lucas Theorem.

\begin{LMA} \label{coeffs}
Let $f(\z,t)\in\CC[\z,t]$ be stable, and let
$$
f(\z,t) = \sum_{k=0}^d f_k(\z) t^k
$$
with $f_d(\z)\not\equiv 0$.  Then $f_k(\z)$ is stable for all $0\leq k\leq d = 
\deg_t(f)$.
\end{LMA}
\begin{proof}
Consider any $0\leq k\leq d$.  Clearly $f_k(\z)$ is a constant multiple of $\partial^k f(\z,t)/\partial t^k |_{t=0}$, which is stable by Lemma \ref{basic} (c, g). 
%
%Let $g(\z,t) = \partial^k f(\z,t)/\partial t^k$.
%By Lemma \ref{basic}(g), $g(\z,t)$ is stable of degree $d-k$ in $t$.
%Let $p(\z,t) = t^{d-k} g(\z,-t^{-1})$.  By Lemma \ref{basic}(e), $p(\z,t)$
%is stable.  Let $q(\z,t) = \partial^{d-k} p(\z,t)/\partial t^{d-k}$.
%By Lemma \ref{basic}(g), $q(\z,t)$ is stable of degree $0$ in $t$.
%In fact, $q(\z,t)$ is a nonzero multiple of $f_k(\z)$, so $f_k(\z)$
%is stable.
\end{proof}

Polynomials $g, h \in \RR[\z]$ are in \emph{proper position}, denoted
by $g \ll h$, if the polynomial $h + \i g$ is stable.  This is the
multivariate analogue of interlacing roots for univariate polynomials with only
real roots.

\begin{PROP}[Lemma 2.8 of \cite{BB5} and Theorem 1.6 of \cite{BB4}] \label{HB-O}
Let $g,h\in\RR[\z]$.\\
\textup{(a)}\ Then $h\ll g$ if and only if $g+th\in\RR[\z,t]$ is stable.\\
\textup{(a)}\ Then $ag+bh$ is stable for all $a,b\in\RR$ if and only if
either $h\ll g$ or $g\ll h$.
\end{PROP}
It then follows from Lemma \ref{basic}(d,g) that if $h \ll g$ then both $h$ and $g$
are stable (or identically zero).

\begin{PROP}[Lemma 2.6 of \cite{BB4}] \label{proper}
Suppose that $g\in \RR[\z]$ is stable. Then the sets
$$
\{h \in \RR[\z] : g\ll h \} \quad \mbox{ and }
\quad \{h \in \RR[\z] : h \ll g\}
$$
are convex cones containing $g$.
\end{PROP} 

\begin{PROP}\label{cone}
Let $V$ be a real vector space, $\phi : V^{n}\rightarrow \RR$
a multilinear form, and $e_1,\ldots,  e_n, v_2, \ldots, v_n$ fixed vectors
in $V$. Suppose that the polynomial 
$$
\phi(e_1,  v_2 +z_2e_2, \ldots,  v_n + z_ne_n)
$$
in $\RR[\z]$ is not identically zero. Then the set of all $v_1 \in  V$
for which the polynomial 
$$
\phi(v_1+z_1e_1,  v_2 +z_2e_2, \ldots,  v_n + z_ne_n)
$$ 
is stable is either empty or a convex cone (with apex $0$) containing
$e_1$ and $-e_1$.
\end{PROP} 
\begin{proof}
Let $C$ be the set of all $v_1 \in V$ for which the polynomial
$\phi(v_1+z_1e_1,  v_2 +z_2e_2, \ldots,  v_n + z_ne_n)$
is stable. For $v\in V$ let $F_v =   \phi(v,  v_2 +z_2e_2, \ldots,  v_n + z_ne_n)$.
Since
$$
\phi(v_1+z_1e_1,  v_2 +z_2e_2, \ldots,  v_n + z_ne_n)= F_{v_1} + z_1F_{e_1}, 
$$ 
we have $C = \{v \in V : F_{e_1} \ll F_{v}\}$. Moreover since $F_{\lambda v+ \mu w}
= \lambda F_v + \mu F_w$ it follows from Proposition \ref{proper} that $C$ is a
convex cone provided that $C$ is non-empty. If $C$ is nonempty then
$F_v + z_1F_{e_1}$ is stable for some $v \in V$. But then $F_{e_1}$ is stable,
and so is
$$
(\pm 1 + z_1)F_{e_1} = \phi(\pm e_1+z_1e_1,  v_2 +z_2e_2, \ldots,  v_n + z_ne_n)
$$
which proves that $\pm e_1 \in C$. 
\end{proof}
(Of course, by permuting the indices Proposition \ref{cone} applies to
any index $j$ as well.)

Let $\CC[\z]^\ma$ denote the vector subspace of \emph{multiaffine} polynomials:\
that is, polynomials of degree at most one in each indeterminate.

\begin{PROP}[Theorem 5.6 of \cite{Br}] \label{Delta}
Let $f\in\RR[\z]^\ma$ be multiaffine.  The following are equivalent:\\
\textup{(a)}\ $f$ is real stable.\\
\textup{(b)}\ For all $1\leq i<j\leq n$ and all $\a\in\RR^n$,
$f_i(\a)f_j(\a)-f(\a)f_{ij}(\a)\geq 0$, in which the subscripts denote
partial differentiation.
\end{PROP}

A linear transformation $T:\CC[\z]\goesto\CC[\z]$ \emph{preserves stability}
if $T(f)(\z)$ is stable whenever $f(\z)$ is stable.

\begin{PROP}[Theorem 1.1 of \cite{BB5}] \label{BB}
Let $T:\CC[\z]^{\ma}\goesto\CC[\z]$ be a linear transformation.
Then $T$ preserves stability if and only if either\\
\textup{(a)}\ $T(f) = \eta(f)\cdot p$ for some linear functional
$\eta:\CC[\z]^{\ma}\goesto\CC$ and stable $p\in\CC[\z]$, or\\
\textup{(b)}\ the polynomial $G_T(\z,\w)=T \prod_{j=1}^n(z_j+w_j)$
is stable in $\CC[\z,\w]$.
\end{PROP}

The \emph{complexification} $T_\CC:\CC[\z]\goesto\CC[\z]$ of a linear
transformation $T:\RR[\z]\goesto\RR[\z]$ is defined as follows.
For any $f\in\CC[\z]$ write $f = g + \i h$ with $g,h\in\RR[\z]$, which can be
done uniquely.  Then $T_\CC(f) = T(g)+\i T(h)$. Let also $\hat{T}_\CC= T(g)-\i T(h)$. 

\begin{PROP} \label{R-to-C}
Let $T:\RR[\z]\goesto\RR[\z]$ be a linear transformation and let $f \in \CC|z]$ be a stable polynomial.
If $T$ preserves real stability then either $T_\CC(f)$ is stable, or $\hat{T}_\CC(f)$ is stable. 
\end{PROP}
\begin{proof}
Let $f=g + \i h\in\CC[\z]$ with $g,h\in\RR[\z]$, and assume that $f$ is
stable.  Then $h\ll g$ by definition.  By Proposition \ref{HB-O}(b),
it follows that $ah+bg$ is real stable for all $a,b\in\RR$.  Therefore,
$aT(h)+bT(g)$ is real stable for all $a,b\in\RR$.  By Proposition \ref{HB-O}(b)
again, either $T(h)\ll T(g)$ or $T(g)\ll T(h)$.  By Proposition \ref{HB-O}(a),
either $T(g)+\i T(h)=T_\CC(f)$ is stable, or $T(g)-\i T(h) =
\hat{T}_\CC(f)$ is stable.
%
%Suppose that there are nonzero polynomials $f$ and $g$ such that
%$T_\CC(f)(\z)$ is stable and $T_\CC(g)(\overline{\z})$ is stable.
%\textbf{CONTRADICTION!  HOW?}
\end{proof}

\section{Proof of the MMCPC.}

\subsection{Reduction to Ferrers matrices.}

\begin{LMA}\label{ferrers}
If $\per(z_j+a_{ij})$ is stable for all Ferrers matrices,
then the MMCPC is true.
\end{LMA}
\begin{proof}
If $\per(z_j+a_{ij})$ is stable for all Ferrers matrices, then by
permuting the columns of such a matrix, the same is true for all monotone
column $\{0,1\}$-matrices.  Now let $A=(a_{ij})$ be an arbitrary $n$-by-$n$
monotone column matrix.  We will show that $\per(z_j+a_{ij})$ is stable by
$n$ applications of Proposition \ref{cone}.

Let $V$ be the vector space of column vectors of length $n$.
The multilinear form $\phi$ we consider is the permanent of an $n$-by-$n$
matrix obtained by concatenating $n$ vectors in $V$.
Let each of $e_1,\ldots,e_n$ be the all-ones vector in $V$.

Initially, let $v_1,v_2,\ldots,v_n$ be arbitrary monotone $\{0,1\}$-vectors.
Then $\phi(v_1+z_1e_1,\ldots,v_n+z_ne_n) = \per(J_nZ_n+H)$
for some monotone column $\{0,1\}$-matrix $H$.
One can specialize any number of $v_j$ to the zero vector, and
any number of $z_j$ to $1$, and the result is not identically zero.
By hypothesis, all these polynomials are stable.

Now we proceed by induction.
Assume that if $v_1,\ldots,v_{j-1}$ are the first $j-1$ columns of $A$,
and if $v_j,\ldots,v_n$ are arbitrary monotone $\{0,1\}$-columns, then
$\phi(v_1+z_1e_1,\ldots,v_n+z_ne_n)$ is stable.  (The base case, $j=1$,
is the previous paragraph.)  Putting $v_j=0$ and $z_j=1$, the resulting
polynomial is not identically zero.  By Proposition \ref{cone} (applied to index $j$),
the set of vectors $v_j$ such that $\phi(v_1+z_1e_1,\ldots,v_n+z_ne_n)$ is
stable is a convex cone containing $\pm e_j$.
Moreover, it contains all monotone $\{0,1\}$-columns, by hypothesis.
Now, any monotone column of real numbers can be written as a nonnegative
linear combination of $-e_1$ and monotone $\{0,1\}$-columns, and
hence is in this cone.  Thus, we may take $v_1,\ldots,v_{j-1},v_j$ to be the
first $j$ columns of $A$, $v_{j+1},\ldots,v_n$ to be arbitrary monotone
$\{0,1\}$-columns, and the resulting polynomial is stable.
This completes the induction step.

After the $n$-th step we find that $\per(J_nZ_n+A)$ is stable.
\end{proof}

\subsection{A more symmetrical problem.}

Let $A=(a_{ij})$ be an $n$-by-$n$ Ferrers matrix, and let $\z=\{z_1,\ldots,z_n\}$.
For each $1\leq j\leq n$, let $y_j=(z_j+1)/z_j$, and let
$Y_n=\diag(y_1,\ldots,y_n)$.  The matrix obtained from $J_nZ_n+A$ by factoring $z_j$
out of column $j$ for all $1\leq j\leq n$ is $AY_n+J_n-A=(a_{ij}y_j+1-a_{ij})$.
It follows that
\begin{equation} \label{z-to-y}
\per(z_j+a_{ij})= z_1 z_2 \cdots z_n\cdot\per(a_{ij}y_j+1-a_{ij}).
\end{equation}

\begin{LMA} \label{y-lma}
For a Ferrers matrix $A=(a_{ij})$, $\per(z_j+a_{ij})$ is stable if and
only if $\per(a_{ij}y_j+1-a_{ij})$ is stable.
\end{LMA}
\begin{proof}
The polynomials are not identically zero.
Notice that $\Im(z_j)>0$ if and only if $\Im(y_j)=\Im(1+z_j^{-1})<0$.
If $\per(z_j+a_{ij})$ is stable, then $\per(a_{ij}y_j+1-a_{ij})\neq 0$
whenever $\Im(y_j)<0$ for all $1\leq j\leq n$.  Since this polynomial has
real coefficients, it follows that it is stable.  The converse is similar.
\end{proof}

The set of $n$-by-$n$ Ferrers matrices has the following duality
$A\mapsto A^\vee = J_n - A^{\top}$:\ transpose and exchange
zeros and ones.  That is, $A^\vee = (a_{ij}^\vee)$ in which $a_{ij}^\vee =
1-a_{ji}$ for all $1\leq i,j\leq n$.  Note that $(A^\vee)^\vee=A$.
However, the form of the expression $\per(a_{ij}y_j+1-a_{ij})$ is not
preserved by this symmetry.  To remedy this defect, introduce new indeterminates
$\x=\{x_1,\ldots,x_n\}$ and consider the matrix $B(A)=(b_{ij})$ with entries
$b_{ij} = a_{ij}y_j+(1-a_{ij})x_{i}$ for all $1\leq i,j\leq n$.
For example, if
$$
A = \left[\begin{array}{lllll}
0 & 1 & 1 & 1 & 1 \\
0 & 0 & 1 & 1 & 1 \\
0 & 0 & 1 & 1 & 1 \\
0 & 0 & 0 & 1 & 1 \\
0 & 0 & 0 & 0 & 0
\end{array}\right]
\hspace{5mm}\mathrm{then}\hspace{5mm}
B(A) = \left[\begin{array}{lllll}
x_1 & y_2 & y_3 & y_4 & y_5 \\
x_2 & x_2 & y_3 & y_4 & y_5 \\
x_3 & x_3 & y_3 & y_4 & y_5 \\
x_4 & x_4 & x_4 & y_4 & y_5 \\
x_5 & x_5 & x_5 & x_5 & x_5
\end{array}\right].
$$

For emphasis, we may write $B(A;\x;\y)$ to indicate that the row variables are
$\x$ and the column variables are $\y$.  The matrices $B(A)$ and $B(A^\vee)$ have
the same general form, and in fact
\begin{eqnarray} \label{dual1}
\per(B(A^\vee;\x;\y)) = \per(B(A;\y;\x)).
\end{eqnarray}
Clearly $\per(B(A))$ specializes to $\per(a_{ij}y_j+1-a_{ij})$ by setting
$x_i=1$ for all $1\leq i\leq n$.  We will show that $\per(B(A))$ is stable,
for any Ferrers matrix $A$.  By Lemmas \ref{basic}(d), \ref{y-lma}, and
\ref{ferrers}, this will imply the MMCPC.

\subsection{A differential recurrence relation.}

Next, we derive a differential recurrence relation for polynomials of the form
$\per(B(A))$, for $A$ an $n$-by-$n$ Ferrers matrix.
There are two cases:\ either $a_{nn}=0$ or $a_{nn}=1$.
Replacing $A$ by $A^\vee$ and using (\ref{dual1}), if necessary, we can assume
that $a_{nn}=0$.

\begin{LMA} \label{diff-reln}
Let $A=(a_{ij})$ be an $n$-by-$n$ Ferrers matrix with $a_{nn}=0$,
let $k\geq 1$ be the number of $0$'s in the last column of $A$,
and let $A^\circ$ be the matrix obtained from $A$ by deleting the
last column and the last row of $A$.  Then
$$
\per(B(A)) =
k x_n\, \per(B(A^\circ)) + x_n y_n\, \partial \per(B(A^\circ)),
$$
in which
$$
\partial =
\sum_{i=1}^{n-k}\frac{\partial}{\partial x_i}
+ \sum_{j=1}^{n-1}\frac{\partial}{\partial y_j}.
$$
\end{LMA}
\begin{proof}
In the permutation expansion of $\per(B(A))$ there are two types of terms:\
those that do not contain $y_n$ and those that do. Let $T_\sigma$ be the
term of $\per(B(A))$ indexed by $\sigma\in\S_n$.  For each
$n-k+1\leq i\leq n$, let $\C_i$ be the set of those terms $T_\sigma$ such
that $\sigma(i)=n$;  for a term in $\C_i$ the variable chosen in the last column
is $x_{i}$.  Let $\D$ be the set of all other terms; for a term in $\D$ the
variable chosen in the last column is $y_n$.

For every permutation $\sigma\in\S_n$, let $(i_\sigma,j_\sigma)$ be such
that $\sigma(i_\sigma)=n$ and $\sigma(n)=j_\sigma$, and define
$\pi(\sigma)\in\S_{n-1}$ by putting $\pi(i)=\sigma(i)$ if $i\neq i_\sigma$,
and $\pi(i_\sigma)=j_\sigma$ (if $i_\sigma\neq n$).
Let $T_{\pi(\sigma)}$ be the corresponding term of $\per(B(A^\circ))$.
See Figure 1 for an example.  Informally, $\pi(\sigma)$ is obtained
from $\sigma$, in word notation, by replacing the largest element with the last
element, unless the largest element is last, in which case it is deleted.

\begin{figure}[tb]
$$
\left[\begin{array}{cccccc}
\cdot & \cdot & \Box & \cdot & \cdot & \cdot \\
\Box & \cdot & \cdot & \cdot & \cdot & \cdot \\
\cdot & \cdot & \cdot & \cdot & \cdot & \Box \\
\cdot & \cdot & \cdot & \cdot & \Box & \cdot \\
\cdot & \Box & \cdot & \cdot & \cdot & \cdot \\
\cdot & \cdot & \cdot & \Box & \cdot & \cdot 
\end{array}\right]
\quad \mapsto \quad
\left[\begin{array}{cccccc}
\cdot & \cdot & \Box & \cdot & \cdot &  \\
\Box & \cdot & \cdot & \cdot & \cdot &  \\
\cdot & \cdot & \cdot & \Box & \cdot &  \\
\cdot & \cdot & \cdot & \cdot & \Box &  \\
\cdot & \Box & \cdot & \cdot & \cdot &  \\
 &  &  &  &  &  
\end{array}\right]
$$
\caption{$\sigma= 3\ 1\ 6\ 5\ 2\ 4$ maps to $\pi(\sigma)= 3\ 1\ 4\ 5\ 2$.} 
\end{figure}

For each $n-k+1\leq i\leq n$, consider all permutations $\sigma$ indexing terms
in $\C_i$.  The mapping $T_\sigma \mapsto T_{\pi(\sigma)}$ is a bijection from the
terms in $\C_i$ to all the terms in $\per(B(A^\circ))$.  Also, for each
$\sigma\in\C_i$, $T_\sigma = x_n T_{\pi(\sigma)}$.  Thus, for each
$n-k+1\leq i\leq n$, the sum of all terms in $\C_i$ is $x_n \per(B(A^\circ))$.

Next, consider all permutations $\sigma$ indexing terms in $\D$.
The mapping $T_\sigma\mapsto T_{\pi(\sigma)}$ is $(n-k)$-to-one
from $\D$ to the set of all terms in $\per(B(A^\circ))$, since one
needs both $\pi(\sigma)$ and $i_\sigma$ to recover $\sigma$.
Let $v_\sigma$ be the variable in position $(i_\sigma,j_\sigma)$ of $B(A^\circ)$.
Then $v_\sigma T_\sigma = x_n y_n T_{\pi(\sigma)}$.  It follows that
for any variable $w$ in the set $\{x_1,\ldots,x_{n-k},y_1,\ldots,y_{n-1}\}$,
the sum over all terms in $\D$ such that $v_\sigma=w$ is
$$
x_n y_n \frac{\partial}{\partial w} \per(B(A^\circ)).
$$
Since $v_\sigma$ is any element of the set
$\{x_1,\ldots,x_{n-k},y_1,\ldots,y_{n-1}\}$, it follows that the sum of all terms in
$\D$ is $x_n y_n \partial \per(B(A^\circ))$.

The preceding paragraphs imply the stated formula.
\end{proof}

\subsection{Finally, proof of the MMCPC}

\begin{THM} \label{perBA-stable}
For any $n$-by-$n$ Ferrers matrix $A$, $\per(B(A))$ is stable.
\end{THM}
\begin{proof}
As above, by replacing $A$ by $A^\vee$ if necessary, we may assume that
$a_{1n}=0$.  We proceed by induction on $n$, the base case $n=1$ being trivial.
For the induction step, let $A^\circ$ be as in Lemma \ref{diff-reln}.
By induction, we may assume that $\per(B(A^\circ))$ is stable;\
clearly this polynomial is multiaffine.
Thus, by Lemma \ref{diff-reln}, it suffices to prove that the linear
transformation $T=k+y_n\partial$ maps stable multiaffine polynomials to stable
polynomials if $k\geq 1$.  This operator has the form $T=k+z_{m}\sum_{j=1}^{m-1}
\partial/\partial z_j$ (renaming the variables suitably).  By Proposition \ref{BB}
it suffices to check that the polynomial
\showon
G_T(\z,\w)
&=& T \prod_{j=1}^m(z_j+w_j)\\
&=& \left(k+z_m\sum_{j=1}^{m-1}\frac{1}{z_j+w_j}\right) \prod_{j=1}^m(z_j+w_j)
\showoff
is stable.
If $z_j$ and $w_j$ have positive imaginary parts for all $1\leq j\leq m$ then
$$
\xi = \frac{k}{z_m} + \sum_{j=1}^{m-1}\frac{1}{z_j+w_j}
$$
has negative imaginary part (since $k\geq 0$).  Thus $z_m\xi\neq 0$.  Also, $z_j+w_j$
has positive imaginary part, so that $z_j+w_j\neq 0$ for each $1\leq j\leq m$.
It follows that $G_T(\z,\w)\neq 0$, so that $G_T$ is stable, completing the
induction step and the proof.
\end{proof}

\begin{proof}[Proof of the MMCPC]
Let $A$ be any $n$-by-$n$ Ferrers matrix.  By Theorem \ref{perBA-stable},
$\per(B(A))$ is stable.  Specializing $x_i=1$ for all $1\leq i\leq n$,
Lemma \ref{basic}(d) implies that $\per(a_{ij}y_j+1-a_{ij})$ is stable.
Now Lemma \ref{y-lma} implies that $\per(z_j+a_{ij})$ is stable.
Finally, Lemma \ref{ferrers} implies that the MMCPC is true.
\end{proof}
Henceforth, we shall refer to the MMCPT -- ``T'' for ``Theorem''.

\section{Further results.}

\subsection{Generalization to rectangular matrices.}

We can generalize Theorem \ref{perBA-stable} to rectangular matrices,
as follows.  Let $A=(a_{ij})$ be an $m$-by-$n$ matrix.
As in the square case, $A$ is a \emph{Ferrers matrix} if it is a
$\{0,1\}$-matrix that is weakly decreasing down columns and increasing across rows.
The matrix $B(A)$ is constructed just as in the square case:
$B(A)=(b_{ij}) = (a_{ij}y_j+(1-a_{ij})x_{i})$, using row variables
$\x=\{x_1,\ldots,x_m\}$ and column variables $\y=\{y_1,\ldots,y_n\}$.
For emphasis, we may write $B(A;\x;\y)$
to indicate that the row variables are $\x$ and the column variables are $\y$.
The symmetry $A\mapsto A^\vee$ takes an $m$-by-$n$ Ferrers matrix to
an $n$-by-$m$ Ferrers matrix. 

Now let $k\leq \min\{m,n\}$.  The \emph{$k$-permanent} of
an $m$-by-$n$ matrix $H=(h_{ij})$ is
$$
\per_k(H) = \sum_R \sum_C \sum_\beta \prod_{i\in R} h_{i,\beta(i)},
$$
in which $R$ ranges over all $k$-element subsets of $\{1,\ldots,m\}$,
$C$ ranges over all $k$-element subsets of $\{1,\ldots,n\}$, and
$\beta$ ranges over all bijections from $R$ to $C$.  (Note that
$\per_0(H)=1$ for any matrix $H$.) In the case $k=m=n$
this reduces to the permanent of a square matrix. 

For an $m$-by-$n$ Ferrers matrix $A$ and $k\leq\min\{m,n\}$, note that
\begin{eqnarray} \label{dual2}
\per_k(B(A^\vee;\x;\y)) = \per_k(B(A;\y;\x)).
\end{eqnarray}
Thus, replacing $A$ by $A^\vee$ if necessary, we may assume that $m\leq n$.

\begin{PROP} \label{m-to-k}
Let $A=(a_{ij})$ be an $m$-by-$n$ Ferrers matrix.
Then $\per_k(B(A))$ is stable for all $k\leq \min\{m,n\}$.
\end{PROP}
\begin{proof}
Using (\ref{dual2}), if necessary, we may assume that $m\leq n$.

We begin by showing that $\per_m(B(A))$ is stable.
Let $A'$ be the $n$-by-$n$ Ferrers matrix obtained by concatenating $n-m$ rows
of $0$'s to the bottom of $A$.
One checks that
$$
\per(B(A'))= (n-m)!\,x_n x_{n-1}\cdots x_{m+1} \cdot \per_m(B(A)).
$$
By Theorem \ref{perBA-stable}, $\per(B(A'))$ is stable, and it follows
easily that $\per_m(B(A))$ is stable.

Now, let $J_{m,n}$ be the $m$-by-$n$ matrix of all 1's.  Then
$$
\per_m(B(A)+tJ_{m,n})
= \sum_{k=0}^m \per_k(B(A)) \binom{n-k}{m-k}(m-k)! t^{m-k}.
$$
By Lemma \ref{basic}(c,f), this polynomial is stable.
Extracting the coefficient of $t^{m-k}$ from this, and dividing by
$ \binom{n-k}{m-k}(m-k)!$, Lemma \ref{coeffs} shows that $\per_k(B(A))$
is stable for all $0\leq k\leq m$.
\end{proof}

Proposition \ref{m-to-k} suggests the idea of a similar generalization of
the MMCPT:\ is it true that $\per_k(J_{m,n}Z_n+A)$ is stable for every
$m$-by-$n$ monotone column matrix $A$ and $k\leq\min\{m,n\}$?
This conjecture originates in \cite{HV}.  One cannot
derive this from Proposition \ref{m-to-k}, however, because there is no
analogue of (\ref{z-to-y}) for $k$-permanents.  Nonetheless, we can prove
this result for half the cases.

\begin{PROP}
Let $A$ be an $m$-by-$n$ monotone column matrix with $m\geq n$, and let
$k\leq n$.  Then $\per_k(J_{m,n}Z_n+A)$ is stable.
\end{PROP}
\begin{proof}
Let $A'$ be the $m$-by-$m$ matrix obtained from $A$ by concatenating
$m-n$ zero columns to the right of $A$.  Then
$$
\per(J_mZ_m+A') = (m-n)!z_m z_{m-1} \cdots z_{n+1}\cdot \per_n(J_{m,n}Z_n+A).
$$
Since $\per(J_mZ_m+A')$ is stable, it follows that $\per_n(J_{m,n}Z_n+A)$
is stable.  By Lemma \ref{basic}(c,f), it follows that
$\per_n(J_{m,n}Z_n+A+tJ_{m,n})$ is stable.
Extracting the coefficient of $t^{n-k}$ from this, and dividing by
$ \binom{m-k}{n-k}(n-k)!$, Lemma \ref{coeffs} shows that $\per_k(J_{m,n}Z_n+A)$
is stable for all $0\leq k\leq n$.
\end{proof}

\subsection{A cycle-counting generalization.}

Theorem \ref{perBA-stable} can be generalized in another direction, as follows.

For each permutation $\sigma\in\S_n$, let $\cyc(\sigma)$ denote the number of
cycles of $\sigma$.  For an indeterminate $\alpha$ and an
$n$-by-$n$ matrix $H=(h_{ij})$, the \emph{$\alpha$-permanent} of $H$ is
$$
\per(H;\alpha) = \sum_{\sigma\in\S_n} \alpha^{\cyc(\sigma)}
\prod_{i=1}^n h_{i,\sigma(i)}.
$$

The numbers $\cyc(\sigma)$ behave well with respect to the duality
$A\mapsto A^\vee$:\ for any Ferrers matrix $A$,
\begin{eqnarray} \label{dual3}
\per(B(A^\vee;\x;\y);\alpha) = \per(B(A;\y;\x);\alpha).
\end{eqnarray}

\begin{LMA} 
Let $A=(a_{ij})$ be an $n$-by-$n$ Ferrers matrix with $a_{nn}=0$,
let $k\geq 1$ be the number of $0$'s in the last column of $A$,
and let $A^\circ$ be the matrix obtained from $A$ by deleting the
last column and the last row of $A$.  Then
$$
\per(B(A);\alpha) = (\alpha + k-1) x_n\, \per(B(A^\circ);\alpha)
+ x_n y_n\, \partial \per(B(A^\circ);\alpha),
$$
with $\partial$ as in Lemma \ref{diff-reln}.
\end{LMA}
\begin{proof}
Adopt the notation of the proof of Lemma \ref{diff-reln}.  To obtain
the present result, observe that if $T_\sigma$ is in $\C_n$ then
$\cyc(\sigma) = 1 + \cyc(\pi(\sigma))$, and otherwise
$\cyc(\sigma)=\cyc(\pi(\sigma))$.
\end{proof}

\begin{PROP} \label{alpha}
For $\alpha> 0$ and $A$ a Ferrers matrix, $\per(B(A);\alpha)$ is stable.
\end{PROP}
\begin{proof}
Reprising the proof of Theorem \ref{perBA-stable}, it suffices to show
that an operator of the form
$$
T=(\alpha+k-1)+z_{m}\sum_{i=1}^{m-1} \partial/\partial z_i
$$
preserves stability when $k\geq 1$.  The argument of the proof of Theorem \ref{perBA-stable} works
when $\alpha> 0$.
\end{proof}

For $\alpha> 0$ and $A$ a Ferrers matrix, specialize all $x_i=1$ and
diagonalize all $y_j=z$ in $\per(B(A);\alpha)$.  By Lemma \ref{basic}(c,d),
the result is a (univariate) polynomial with only real roots.  This special
case is also implied by Theorem 2.5 of \cite{H}.

%We also get a cycle counting version of MMCPT, the of which is identical to that of MMCPT. 
%
%\begin{THM}
%Suppose that $A$ is a column monotone $n$-by-$n$ matrix and that $\alpha\geq 0$. Then 
%$$
%\per(J_nZ_n+A,\alpha)
%$$
%is stable. 
%\end{THM}

\subsection{Multivariate stable Eulerian polynomials.}

Given a permutation $\sigma \in \S_n$, viewed as a linear sequence
$\sigma (1) \sigma (2) \cdots \sigma (n)$, 
let $\text{L}(\sigma)$ denote the result of the following procedure.  First  
form the two-line array 
\showon
\text{T}(\sigma) = \left (
\begin{matrix}
1 & 2 & \cdots & n \\
\sigma (1) & \sigma (2) & \cdots & \sigma (n)
\end{matrix}
\right ).
\showoff
Then, viewing $\text{T}(\sigma)$ as a map sending $i$ to $\sigma (i)$,
break $\text{T}(\sigma)$ into cycles, 
with the smallest element of each cycle at the end, and 
cycles listed left-to-right with smallest elements increasing.
Finally, erase the parentheses delimiting the cycles to form
a new linear sequence $\text{L}(\sigma)$.  For example, 
if $\sigma=341526978$, then $\text{T}(\sigma)$ has cycle decomposition
$(31)(452)(6)(987)$ and $\text{L}(\sigma)=314526987$.   

Let $\text{P}(\sigma)$ denote the placement of $n$ nonattacking rooks on
the squares $(i,\sigma (i))$, $1\le i \le n$.
As noted by Riordan \cite{Rio}, rooks in $\text{P}(\sigma)$ that occur above the
diagonal correspond to 
exceedences (values of $i$ for which $\sigma (i) > i$) in $\sigma$ and descents 
(values of $i$ for which $\sigma (i) > \sigma (i+1)$) in $\text{L}(\sigma)$.  
Hence 
\begin{eqnarray}
\label{Eulerian}
\sum_{\sigma \in \S_n} 
z^{\text{exc}(\sigma)} =
\sum_{\sigma \in \S_n} 
z^{\text{des}(\sigma)},
\end{eqnarray}
where $\text{exc}(\sigma)$ is the number of exceedences and
$\text{des}(\sigma)$ the number of descents
of $\sigma$.

The polynomial in (\ref{Eulerian}) is known as the {\it Eulerian polynomial}.
It is one of the classic examples in
combinatorics of a polynomial with only real roots.
Let $E_n=(e_{ij})$ denote the $n$-by-$n$ matrix with $e_{ij}=0$ if $i\geq j$
and $e_{ij}=1$ if $i<j$.  Then $\per(B(E_n;\x;\y))$ is stable by Theorem
\ref{perBA-stable}. Let $\one$ be the vector (of appropriate size) of all ones. 
From the above discussion, $\per(B(E_n;\one;z,\ldots,z))$ equals the
Eulerian polynomial; by Lemma \ref{basic}(c,d) this is a univariate stable
polynomial with real coefficients, so it has only real roots.
Similarly, we see that
\begin{align}
\label{BBB}
\per(B(E_n;\one;\y)) &= 
\sum_{\sigma \in \S_n} 
\prod _{ \sigma(i) > \sigma(i+1)} y_{\sigma(i)} \\ &=
\sum_{\sigma \in \S_n} 
\prod _{ \sigma(i) > i} y_{\sigma(i)}
\end{align}
is stable in $\{y_1,y_2,\ldots ,y_{n}\}$ (but $y_1$ does not really occur).  

Letting $f=\per(B(E_n;\one;\y))$, note that the partial of $f$ with respect to $y_i$,
evaluated at all $y$-variables equal to $1$, equals the number of permutations
in $\S_n$ that have $i$ as a ``descent top", i.e. have the property that
$i$ is followed immediately
by something less than $i$.  Denoting this number by $\text{Top}(i;n)$,  
applying Proposition \ref{Delta} to $f$ we get
\begin{align}
\text{Top}(i;n)\text{Top}(j;n) \ge n! \, \text{Top}(i,j;n),
\end{align}
where $\text{Top}(i,j;n)$ is the number of permutations in $\S_n$ 
having both $i$ and $j$ as descent
tops, with $2\le i<j\le n$.  Dividing both sides of the above equation by $n!^2$ shows that
occurrences of descent tops in a uniformly random permutation are negatively correlated.

More general forms of (\ref{BBB}) can be defined which still maintain stability.  
First of all,
cycles in $\text{T}(\sigma)$
clearly translate into left-to-right minima in $\text{L}(\sigma)$, and so by Proposition
\ref{alpha} the polynomial 
\showon
\sum_{\sigma \in \S_n} 
\alpha ^{\text{LRmin}(\sigma)}
\prod _{\sigma(i) > \sigma(i+1)} y_{\sigma(i)} 
\showoff
is stable in the $\y$ for $\alpha>0$, with $\text{LRmin}(\sigma)$
denoting the number of left-to-right minima of $\sigma$.

Secondly, the sum in (\ref{BBB}) can be replaced by a sum over permutations of a
multiset.  For a given vector $\v =(v_1,\ldots ,v_t) \in \mathbb N ^t$, 
let $N(\v) = \{1^{v_1}2^{v_2}\cdots t^{v_t}\}$ denote the multiset
with $v_i$ copies of $i$, and $\M(\v)$ the set of multiset permutations
of $N(\v)$ (so for example $\M(2,1) = \{112,121,211\}$).
Riordan \cite{Rio} noted that if we map our previous sequence $L(\sigma)$ to a multiset permutation 
$m(\sigma)$ by replacing numbers $1$ through $v_1$ in $L(\sigma)$ by $1$'s, 
numbers $v_1+1$ through $v_1+v_2$ by $2$'s, etc., we get a $\prod v_i!$ to $1$ map, 
and furthermore certain squares above the diagonal where rooks in $P(\sigma)$ 
correspond to descents in $L(\sigma)$ no longer correspond to descents in  
$m(\sigma)$. For example, if $v_1=2$, then $1$ and $2$ in $L(\sigma)$ both get mapped 
to $1$ in $m(\sigma)$, so a rook on square $(1,2)$ no longer corresponds to a descent.  

Let $n$ denote the sum of the coordinates of $\v$, and let 
$Y(\v)$ be the sequence of variables obtained by starting with 
$y_1,\ldots ,y_n$ and setting the first $v_1$ $y$-variables equal to $y_1$, 
the next $v_2$ $y$-variables equal to $y_2$, etc. Then if $E(\v)$ is the 
Ferrers matrix whose first $v_1$ columns are all zeros, the next $v_2$ columns have 
ones in the top $v_1$ rows and zeros below, the next $v_3$ columns have ones in the 
top $v_1+v_2$ rows and zeros below, etc., an easy extension of the argument 
above implies 
\begin{align}
\label{MBBB}
(1/\prod _i v_i !) \text{per}(   B(  E(\v;\one;Y(\v))  )   )
=\sum_{\sigma \in \M(\v)} 
\prod _{ \sigma(i) > \sigma(i+1)} y_{\sigma(i)}
\end{align}
is stable in the $y_i$. 
This contains Simion's result \cite{Sim}, that the multiset Eulerian polynomial 
has only real roots.  If $\v$ has all ones, i.e. $N(\v)$ is a set, it reduces to our 
previous result.  Finally, we note that this argument also shows that if we replace the
condition $\sigma (i) > \sigma (i+1)$ by the more general condition 
$\sigma (i) > \sigma (i+1)+j-1$, for any fixed positive integer $j$, we still get
stability.

\subsection{Grace's Apolarity Theorem.}

Univariate complex polynomials $f(t)=\sum_{k=0}^n \binom{n}{k} a_k t^k$ and
$g(t)=\sum_{k=0}^n \binom{n}{k} b_k t^k$ are \emph{apolar} if $a_nb_n\neq 0$
and
$$
\sum_{k=0}^n \binom{n}{k} (-1)^{n-k} a_k b_{n-k} = 0. 
$$

\begin{LMA} \label{per-wz}
Let $f(t)=\sum_{k=0}^n \binom{n}{k} a_k t^k$ and
$g(t)=\sum_{k=0}^n \binom{n}{k} b_k t^k$ be complex polynomials of degree $n$.
Let the roots of $f(t)$ be $z_1,\ldots,z_n$ and let the roots of $g(t)$ be
$w_1,\ldots,w_n$.  Then
$$
\sum_{k=0}^n \binom{n}{k} (-1)^{n-k} a_k b_{n-k} = n!\,a_n b_n\, \per(w_i-z_j). 
$$
In particular, $f(t)$ and $g(t)$ are apolar if and only if $\per(w_i-z_j)=0$.
\end{LMA}
\begin{proof}
It suffices to prove this for monic polynomials $f(t)$ and $g(t)$.
For each permutation $\sigma\in\S_n$ there are $2^n$ terms in $\per(w_i-z_j)$,
since for each $1\leq i\leq n$ either $w_i$ or $-z_{\sigma(i)}$ can be chosen.
For each subset $R$ of rows of size $k$, and subset $C$ of columns of size
$n-k$, the monomial $\w^R\z^C$ is produced $k!(n-k)!$ times.
Since $(-1)^{k}a_{n-k}$ is the $k$-th elementary symmetric function of
$\{z_1,\ldots,z_n\}$, and similarly for $(-1)^{k}b_{n-k}$ and $\{w_1,\ldots,w_n\}$,
the result follows.
\end{proof}

\begin{LMA} \label{mobius}
Let $f(t)$ and $g(t)$ be polynomials of degree $n$.
Let $t\mapsto\phi(t)=(at+b)/(ct+d)$ be a M\"obius transformation,
with inverse $\phi^{-1}(t)=(\alpha t+\beta)/(\gamma t + \delta)$.
Let $\widehat{f}(t)=(\gamma t+\delta)^n f(\phi^{-1}(t))$ and
$\widehat{g}(t)=(\gamma t+\delta)^n g(\phi^{-1}(t))$ have degree $n$.
Then $\widehat{f}(t)$ and $\widehat{g}(t)$ are apolar if and only if
$f(t)$ and $g(t)$ are apolar.
\end{LMA}
\begin{proof}
Let the roots of $f(t)$ be $z_1,\ldots,z_n$ and let the roots of $g(t)$
be $w_1,\ldots,w_n$.  Then the roots of $\widehat{f}(t)$ are
$\phi(z_1),\ldots,\phi(z_n)$ and the roots of $\widehat{g}(t)$ are
$\phi(w_1),\ldots,\phi(w_n)$.  Consider the permanent $\per(\phi(w_i)-\phi(z_j))$.
The $(i,j)$-entry of this matrix is
$$
\frac{aw_i+b}{cw_i+d}-\frac{az_j+b}{cz_j+d}=\frac{(ad-bc)(w_i-z_j)}{(cw_i+d)(cz_j+d)}.
$$
Factor $(cw_i+d)^{-1}$ out of row $i$, factor $(cz_j+d)^{-1}$ out of column
$j$, and factor $ad-bc$ out of every row.  Therefore
$$
\per(\phi(w_i)-\phi(z_j)) = \frac{(ad-bc)^n}{\prod_{h=1}^n (cw_h+d)(cz_h+d)}
\cdot \per(w_i-z_j).
$$
Since the prefactor on the right-hand side is neither zero nor infinite,
the result follows from Lemma \ref{per-wz}.
\end{proof}

A \emph{circular region} $\A$ is a proper subset of $\CC$ that is either
open or closed, and is bounded by either a circle or a straight line.

\begin{THM}[Grace's Apolarity Theorem] \label{grace}
Let $f(t)$ and $g(t)$ be apolar polynomials.  If every root of $g(t)$
is in a circular region $\A$, then $f(t)$ has at least one root in $\A$.
\end{THM}
\begin{proof}
It suffices to prove this for monic polynomials $f(t)$ and $g(t)$, and
for open circular regions $\A$ since a closed circular region is the intersection
of all open circular regions which contain it.  Let $t\mapsto\phi(t)$ be
a M\"obius transformation that maps $\A$ to the upper half-plane $\HH$.
By Lemma \ref{mobius}, it suffices to prove this when the circular region
is $\HH$ itself.  Let the roots of $f(t)$ be
$z_1,\ldots,z_n$ and let the roots of $g(t)$ be $w_1,\ldots,w_n$.

If all of $w_1,\ldots,w_n$ are real numbers, then by permuting the rows of
$(w_i+z_j)$ we can assume that $w_1\leq\cdots\leq w_n$ without changing the
value of $\per(w_i+z_j)$.  By the MMCPT,
$\per(w_i+z_j)$ is a real stable polynomial in $z_1,\ldots,z_n$.
In other words, the transformation $T:\RR[z]\goesto\RR[\z]$ defined by
$$
T(f(z)) = \per(w_i+z_j), 
$$
where $f(z)=(z+w_1)(z+w_2)\cdots(z+w_n)$ 
preserves real stability.  This is a linear transformation.  Suppose that $f(z) \in \CC[z]$ is stable. By Proposition
\ref{R-to-C}, either $T_\CC(f(z))$ is stable or $\hat{T}_\CC(f(z))$ is stable.
Diagonalizing by setting $z_j=z$ for all $1\leq j\leq n$, we see that
$T_\CC(f)(z,\ldots,z)= n! f(z)$, so that $T_\CC(f(z))$ is stable.
Therefore $\per(w_i+z_j)$ is a stable polynomial
in $\CC[\z]$, for all $w_1,\ldots,w_n\in\HH$.  Therefore $\per(w_i+z_j)$
is a stable polynomial in $\CC[\w,\z]$. Actually it satisfies a stronger stability property. Namely if $z_j \in \overline{\HH}$ and 
$w_j \in \HH$ for all $1\leq j\leq n$, then $\per(w_i+z_j) \neq 0$. Indeed if we fix $\zeta_j \in \overline{\HH}$ for all $1\leq j\leq n$, then 
the polynomial $\per(w_i + \zeta_j) \in \CC[\w]$ is either identically zero or stable (and not identically zero) by Lemma \ref{basic} (d). Clearly  $\per(w_i + \zeta_j)$ is not identically zero.  
 
Now assume that $w_i\in\HH$ for all $1\leq i\leq n$.  Arguing for a
contradiction, assume that $z_j\not\in\HH$ for all $1\leq j\leq n$.
Then $-z_j\in\overline{\HH}$ for all $1\leq j\leq n$.  Hence  $\per(w_i-z_j)\neq 0$ by the argument given in the previous paragraph.  But $\per(w_i-z_j)=0$ by Lemma
\ref{per-wz}, since $f(t)$ and $g(t)$ are apolar.  This contradiction completes
the proof.
\end{proof}

Our proof of Grace's apolarity theorem relies on MMCPT. To avoid going in circles we should ensure ourselves that the 
proof of MMCPT does not use any form of Grace's apolarity theorem. In fact what we use in the proof of MMCPT is that condition (b) in Proposition \ref{BB} implies that the operator preserves stability, the proof of which does not use Grace's apolarity theorem. 

\subsection{Permanental inequalities}
If $A$ is a $n$-by-$n$ matrix and $S \subseteq [n]=\{1,\ldots,n\}$ let $A_S$ be the matrix obtained by replacing the columns indexed by $S$ by columns of all ones. 
\begin{CORO}\label{negass}
Let $A$ be an  $n$-by-$n$ monotone matrix, $S \subset [n]$,  and  $i,j \in [n]\setminus S$. Then 
$$
\per(A_{S\cup \{i\}})\per(A_{S\cup \{j\}}) \geq \per(A_{S\cup \{i,j\}}) \per(A_S). 
$$
\end{CORO}
\begin{proof}
The proof follows immediately from Proposition \ref{Delta} applied to $\per(z_j+a_{ij})$. 
\end{proof}
Note that Corollary \ref{negass} can be interpreted as that the permanent is pairwise negatively associated in columns. 

The \emph{generating polynomial} of a discrete measure $\mu : 2^{[n]} \rightarrow \RR_{\geq 0}$ is given by 
$$
G_\mu(\z)= \sum_{S \in 2^{[n]}} \mu(S) \prod_{j \in S}z_j.
$$
The measure $\mu$ is \emph{Rayleigh} if 
\begin{equation}\label{Ra}
\frac {\partial G_\mu}{\partial z_i}(\x)\frac {\partial G_\mu}{\partial z_j}(\x) \geq \frac {\partial^2 G_\mu}{\partial z_i\partial z_j}(\x)G_\mu(\x),
\end{equation}
for all $\x \in \RR_{\geq 0}^n$, and it is called \emph{strongly Rayleigh} if \eqref{Ra} holds for all $\x \in \RR^n$. We refer  
to \cite{BBL,W1} for more information on Rayleigh and  strongly Rayleigh measures. 

Suppose that $A$ is an  $n$-by-$n$ monotone matrix with nonnegative coefficients. Consider the 
discrete measure $\mu_A : 2^{[n]} \rightarrow \RR_{\geq 0}$ defined by $\mu_A(S) = \per(A_S)$. By Proposition \ref{Delta} we see that $\mu_A$ is strongly Rayleigh. This fact entails many inequalities. 
\begin{CORO}\label{nlc}
Suppose that $A$ is an  $n$-by-$n$ monotone matrix with nonnegative coefficients. Then 
$$
\per(A_S)\per(A_T) \geq \per(A_{S\cup T})\per(A_{S \cap T}), 
$$
for all $S, T \subseteq [n]$. 

Moreover
$$
\per(A) \leq s_1\cdots s_n \frac {n!}{n^n}, 
$$
where $s_i$ is the sum of the elements in the $i$th column. 
\end{CORO}
\begin{proof}
The first inequality holds for all Rayleigh measures, see \cite[Theorem 4.4]{W1}. 

Let $\mu(S)= \per(A_{[n]\setminus S})/n!$. By the above $
\mu(S)\mu(T) \geq \mu(S\cup T)\mu(S \cap T) 
$
for all $S, T \subseteq [n]$. Thus $\mu(S \cup T) \leq \mu(S)\mu(T)$ whenever $S\cap T=\emptyset$, and after iteration 
$\mu([n]) \leq \mu(\{1\})\cdots \mu(\{n\})$. The proof now follows by observing that 
$\mu(\{i\})= s_i/n$ for all $i \in [n]$.  

One can also prove the last inequality by an elementary argument. If there are two 
different consecutive elements $b>a$ in a column of $A$, replace these by their average to obtain the matrix $A'$. It is plain to see that $\per(A) \leq \per(A')$. Iterating this procedure it follows that $\per(A) \leq \per(B)$, where each element in column $i$ of $B$ is equal to $s_i/n$. The inequality now follows.
\end{proof}
The second inequality in Corollary \ref{nlc} can be compared with the Van der Waerden conjecture
which asserts that the permanent of a doubly stochastic matrix is greater than or equal to $n!/n^n$. The Van der Waerden conjecture which was stated in 1926 was proved by Falikman in 1981; the case of equality was proved by Egorychev. Gurvits has recently provided a beautiful proof of a vast generalization of  the Van der Waerden conjecture using stable polynomials, see \cite{Gu} and 
\cite[Section 8]{W}.

If $f : 2^{[n]} \rightarrow \RR$, and $\mu$ is a discrete measure on $2^{[n]}$ we let 
$$
\int f d \mu = \sum_{S \in 2^{[n]}} f(S)\mu(S).
$$
A measure $\mu$ on $2^{[n]}$ is \emph{negatively associated} if for all increasing functions 
$f,g : 2^{[n]} \rightarrow \RR$ depending on disjoint sets of variables 
$$
\int fg d\mu \int d\mu \leq \int f d\mu \int g d\mu, 
$$
see e.g. \cite{BBL}. 
\begin{CORO}
Suppose that $A$ is an  $n$-by-$n$ monotone matrix with nonnegative coefficients.  Then the  
discrete measure $\mu_A : 2^{[n]} \rightarrow \RR_{\geq 0}$, defined by $\mu_A(S) = \per(A_S)$, is negatively associated. 
\end{CORO}
\begin{proof}
The corollary follows from the fact that $\mu_A$ is strongly Rayleigh. Such measures are negatively associated, which was proved in \cite[Theorem 4.9]{BBL}
\end{proof}

\end{document}